\documentclass{amsart}

\usepackage[utf8]{inputenc}

\usepackage{geometry}
\usepackage{graphicx}

\usepackage{amsthm}
\usepackage{amssymb}
\usepackage{amsmath}
\usepackage{comment}
\usepackage{thm-restate}
\usepackage{url} 
\usepackage{hyperref}
\usepackage[noabbrev,capitalise]{cleveref}
\usepackage{dsfont}

\usepackage{makecell}

\usepackage{multicol}

\usepackage{enumitem}

\usepackage{color}
\usepackage[normalem]{ulem}

\usepackage{tikz}

\usepackage{caption}
\usepackage[labelformat=simple]{subcaption}

\theoremstyle{plain}
\newtheorem{thm}{Theorem}[section]
\newtheorem{lem}[thm]{Lemma}
\newtheorem{proposition}[thm]{Proposition}

\newtheorem{obs}[thm]{Observation} 

\newtheorem{conj}[thm]{Conjecture}

\crefname{thm}{Theorem}{Theorems}
\crefname{lem}{Lemma}{Lemmas}

\counterwithin*{claim}{section}

\theoremstyle{definition}

\newcommand{\R}{\mathbb{R}}

\newcommand{\ex}{\textup{ex}}

\usepackage{algorithm}
\usepackage{algpseudocode}

\allowdisplaybreaks

\makeatletter
\@namedef{subjclassname@2020}{%
\textup{2020} Mathematics Subject Classification}
\makeatother

\makeatletter
    \def\@@and{}
\makeatother

\usepackage{todonotes}

\title{Longest cycles in vertex-transitive and highly connected graphs}

\author[C. Groenland]{Carla Groenland}
\address{Delft Institute of Applied Mathematics, Technische Universiteit Delft, 2628 CD Delft, Netherlands}
\email{c.e.groenland@tudelft.nl}
\urladdr{carlagroenland.com} 

\author[S. Longbrake]{Sean Longbrake}
\address{Department of Mathematics, Emory University, Atlanta, United States}
\email{sean.longbrake@emory.edu}
\urladdr{www.seanlongbrake.org}

\author[R. Steiner]{Raphael Steiner}
\address{Department of Computer Science, Institute of Theoretical Computer Science, ETH Z\"{u}rich, Switzerland}
\email{raphaelmario.steiner@inf.ethz.ch}
\urladdr{sites.google.com/view/raphael-mario-steiner}

\author[J. Turcotte]{J\'er\'emie Turcotte}
\address{Department of Mathematics and Statistics, McGill University, Montr\'eal, Canada}
\email{mail@jeremieturcotte.com}
\urladdr{www.jeremieturcotte.com}

\author[L. Yepremyan]{Liana Yepremyan}
\address{Department of Mathematics, Emory University, Atlanta, United States}
\email{liana.yepremyan@emory.edu}
\urladdr{www.lianayepremyan.com}

\thanks{The research of the third author was funded by the Swiss national Science Foundation Ambizione Grant No. 216071. The fourth author is supported by the Natural Sciences and Engineering Research Council of Canada (NSERC) and the Fonds de Recherche du Québec – Nature et technologies (FRQNT). Le quatrième auteur est supporté par le Conseil de recherches en sciences naturelles et en génie du Canada (CRSNG) et le Fonds de Recherche du Québec – Nature et technologies (FRQNT). The research of the fifth author is supported by the National Science Foundation grant DMS-2247013: Forbidden and Colored Subgraphs.}


\begin{document}

\begin{abstract}
We present progress on three old conjectures about longest paths and cycles in graphs. The first pair of conjectures, due to Lov\'{a}sz from 1969 and Thomassen from 1978, respectively, states that all connected vertex-transitive graphs contain a Hamiltonian path, and that all sufficiently large such graphs even contain a Hamiltonian cycle. The third conjecture, due to Smith from 1984, states that for $r\ge 2$ in every $r$-connected graph any two longest cycles intersect in at least $r$ vertices. In this paper, we prove a new lemma about the intersection of longest cycles in a graph which can be used to improve the best known bounds towards all the aforementioned conjectures: 
First, we show that every connected vertex-transitive graph on $n\geq 3$ vertices contains a cycle (and hence path) of length at least $\Omega(n^{13/21})$, improving on $\Omega(n^{3/5})$ from [DeVos, \emph{arXiv:2302:04255}, 2023]. 
Second, we show that in every $r$-connected graph with $r\geq 2$, any two longest cycles meet in at least $\Omega(r^{5/8})$ vertices, improving on $\Omega(r^{3/5})$ from [Chen, Faudree and Gould, \emph{J. Combin. Theory, Ser.~ B}, 1998]. Our proof combines combinatorial arguments, computer-search and linear programming. 
\end{abstract}

\maketitle

\section{Introduction}
One of the oldest and most fundamental research areas in graph theory is concerned with \emph{longest cycles}\footnote{Throughout this paper, a longest cycle in a graph will be defined as a cycle attaining the maximum possible length among all cycles in the graph.} in graphs. In particular, sufficient conditions for the Hamiltonicity of a graph, that is, the existence of a cycle of length $n$ in a graph on $n$ vertices, have been widely studied, going back to the classical result of Dirac~\cite{Dirac52} stating that every graph with minimum degree at least $n/2$ is Hamiltonian for all $n\geq 3$.  The minimum degree condition is tight in Dirac's theorem, so to obtain similar results for sparser graphs one needs to add additional constraints, such as connectivity conditions. For $2$-connected  regular graphs,  the degree condition was improved to $n/3$ by Jackson~\cite{Jackson} and later to the tight $n/3-1$ by Hilbig~\cite{Hilbig}, answering a question of Szekeres (see \cite{Jackson}). For $3$-connected regular graphs, asymptotically confirming a conjecture\footnote{Their conjecture was made for general $t$-connected $d$-regular graphs, $t\geq 3$ and said that  $d\geq n/t+1$ is sufficient to guarantee Hamiltonicity. However, this was disproved for $t\geq 4$ by by Jung~\cite{Jung} and independently  by Jackson, Li and Zhu~\cite{JLZ}.} independently made by Bollob\'as~\cite{Bollobas} and H\"aggkvist~\cite{Haggkvist}, K\"{u}hn, Lo, Osthus and Staden~\cite{KLOS} showed that minimum degree $n/4+o(n)$ is sufficient to guarantee the graph being Hamiltonian. More generally, in the same paper they also considered some sufficient conditions on the length of the longest cycle  of a graph. An old conjecture of Bondy~\cite{Bondy}, confirmed by Wei~\cite{Wei}, says that for all $r\geq 3$, every sufficiently large $2$-connected regular graph on $n$ vertices with degree at least $n/r$ has a cycle of length at least $2n/(r-1)$. K\"{u}hn, Lo, Osthus and Staden~\cite{KLOS} considered a generalization of this conjecture for $t$-connected graphs and showed that degree at least $n/r+o(n)$ guarantees a cycle of length at least $\min\{tn/(r-1), (n-o(n))\}$. 

In contrast to the dense case, sufficient conditions for Hamiltonicity that also address sparse graphs, with, say, a constant average degree, seem to be rather scarce. The positive results that apply to sparse graphs can be roughly clustered into three (mostly orthogonal) research directions: conditions related to pseudo-randomness such as expanders (see e.g.~the recent breakthrough~\cite{Draganic24}); toughness conditions~\cite{Bauer06,Chvatal}; and symmetry conditions such as vertex-transitivity. Note that for dense vertex-transitive  graphs, that is, of degree $d=\Theta(n)$, a lower bound of $(1-o(1))n$ on the maximum cycle length follows from the aforementioned results: Indeed,  every connected vertex-transitive graph of degree $d$ is $2(d + 1)/3$-connected~\cite{GR}, and thus the results from~\cite{KLOS} discussed above imply the existence of a cycle of length $(1-o(1))n$. 

In fact, even stronger results hold: Christofides, Hladk{\`y} and M\'ath\'e~\cite{CHM} showed that all connected vertex-transitive graphs of degree of order $\Theta(n)$ contain a Hamiltonian cycle. Our first main result concerns longest cycles in general connected vertex-transitive graphs.

There are two main open problems in the study of Hamiltonicity on vertex-transitive graphs that are commonly attributed to Lov\'{a}sz and Thomassen, respectively. The first is a conjecture posed by Lov\'asz~\cite{Lovasz69} in 1969 and asserts that every connected vertex-transitive graph admits a Hamiltonian path. The second is a conjecture by Thomassen from 1978 (see e.g.~\cite{Barat10}, Conjecture~1 and the references therein) and states that apart from finitely many exceptional graphs, all connected vertex-transitive graphs even contain a Hamiltonian \emph{cycle}. In fact, to this date only four connected vertex-transitive graphs on at least $3$ vertices are known which do not have a Hamiltonian cycle (one of which is the infamous Petersen graph). 
It is worth mentioning that Lov\'{a}sz originally posed his conjecture in the opposite form, but nowadays the above mentioned positive version has become standard. Curiously, Babai~\cite{Babai95} has proposed a conjecture contradicting Lov\'{a}sz conjecture, namely that there exists an absolute constant $\varepsilon>0$ such that for infinitely many $n$ there exists a vertex-transitive graph of order $n$ with no path of length larger than $(1-\varepsilon)n$.
Despite half a century of study, during which the conjectures of Lov\'{a}sz and Thomassen received much attention by the research community, they remain wide open today. In particular, for 44 years the best known general lower bound on the maximum length of a path or cycle in a connected vertex-transitive graph on $n\ge 3$ vertices remained $\Omega(\sqrt{n})$ due to Babai~\cite{Babai79}, and only recently this bound was improved to $\Omega(n^{3/5})$ by DeVos~\cite{deVos}. There is some further recent evidence of this conjecture being true though: indeed results of Draganić, Montgomery, Munh\'{a} Correia, Pokrovskiy and Sudakov~\cite{Draganic24}, together with earlier results on expansion properties of random Cayley graphs obtained by Alon and Roichman~\cite{AR}, imply that  random Cayley graphs with generator set of size at least $\Omega(\log{n})$ are Hamiltonian. For more related work, we refer to the surveys~\cite{Alspach81, Kutnar09} on the partial progress towards Lov\'{a}sz' and Thomassen's conjectures. As our first main result, we obtain the following improved lower bound for the length of longest cycles in vertex-transitive graphs.

\begin{thm}
\label{thm:long_cycle}
Every connected vertex-transitive graph on $n\geq 3$ vertices contains a cycle of
length at least $\Omega(n^{13/21})$.
\end{thm}

The second main result of our paper addresses a related old conjecture about longest cycles in graphs from 1984, which is commonly attributed to Smith. 

\begin{conj}[cf.~Conjecture 5.2 in~\cite{Groetschel84} and Conjecture 4.15 in~\cite{Bondy95}]
For all $r \ge 2$, every two longest cycles in an $r$-connected graph meet in at least $r$ vertices. 
\end{conj}
Smith's conjecture has been verified for all values $r\le 6$ by Gr\"{o}tschel~\cite{Groetschel84} and for $r\in\{7,8\}$ by Stewart and Thompson~\cite{stewart1995}. Furthermore, Gutiérrez and Valqui~\cite{Gutierrez24} recently proved the conjecture for $r \ge (n+16)/7$. Regarding general values of $r$, Burr and Zamfirescu (see~\cite{Bondy95}) proved that every two longest cycles in an $r$-connected graph must meet in at least $\sqrt{r}-1$ vertices, and this bound was improved to $\Omega(r^{3/5})$ by Chen, Faudree and Gould~\cite{ChenFaudreeGould98} in 1998. For further background on Smith's conjecture and related problems, we refer to the survey article~\cite{Shabbir13} on intersections of longest cycles and paths in graphs. Our second main result gives the first asymptotic improvement of the lower bound towards Smith's conjecture since the result of Chen, Faudree and Gould from 26 years ago.

\begin{thm}
\label{thm:rconnected_corollary}
    There is a constant $C>0$ such that for all $r\geq 2$, every two longest cycles in an $r$-connected graph meet in at least $Cr^{5/8}$ vertices.
\end{thm}

We obtain both our main results, Theorem~\ref{thm:long_cycle} and Theorem~\ref{thm:rconnected_corollary}, from the following key lemma, which can be seen as a ``local'' version of Theorem~\ref{thm:rconnected_corollary}. We denote by $Q_3^+$ the bipartite graph obtained from the 
$3$-dimensional hypercube $Q_3$
by adding one of the diagonals. Alternatively, $Q_3^+$ is the graph obtained from $K_{4,4}$ by removing a matching of size $3$. 

\begin{lem}
    \label{lem:separator_size}
    Let $C_1$ and $C_2$ be two longest cycles in a graph $G$. If $k:=|V(C_1)\cap V(C_2)|>0$ and $S$ is a minimum vertex separator between $A:=V(C_1)\setminus V(C_2)$ and $B:=V(C_2)\setminus V(C_1)$ in $G$, then
$$|S|\leq \textup{ex}(2k,\{Q_3^+,K_{3,3}\})=O(k^{8/5}).$$
\end{lem}

The proof of this lemma is the main bulk of work in this paper. It is partly computer-assisted and will be presented in Sections~\ref{sec:defs} and~\ref{sec:lemmaproof}. 
The upper bound on the extremal function $\textup{ex}(n,\{Q_3^+,K_{3,3}\})=O(n^{8/5})$ follows from the bound $\textup{ex}(n,Q_3^+)\le O(n^{8/5})$ that was first proved by Erd\H{o}s and Simonovits~\cite{ES-cubewithdiagonal}, see also the recent work of F\"{u}redi~\cite{Furedi13} for an independent proof and related results.

Lemma~\ref{lem:separator_size} strengthens a corresponding result by DeVos (see the proof of Lemma~3 in~\cite{deVos}), who proved that $S$  as in Lemma~\ref{lem:separator_size} must satisfy  $|S|=O(k^2)$ instead of $|S|=O(k^{8/5})$; this is obtained by a simple application of Menger's theorem. 

Note that using the upper bound of $\textup{ex}(n,\{Q_3^+,K_{3,3}\})\leq \textup{ex}(n,Q_3^+)$ is not necessarily wasteful, as it may seem at first. Namely, a well-known open conjecture by Erd\H{o}s and Simonovits~\cite{ES} states that for any finite family of graphs $\mathcal{F}$ the Tur\'{a}n number of the family $\mathcal{F}$ is within a constant factor of the minimum among the Tur\'{a}n numbers of members of $\mathcal{F}$ (see Conjecture~1 in ~\cite{ES}). Since it is known that $\textup{ex}(n,K_{3,3})=(\frac{1}{2}+o(1))n^{5/3}$ (see~\cite{brown66, furedi96}), this would imply that $\textup{ex}(n,\{Q_3^+,K_{3,3}\})$ and $\textup{ex}(n,Q_3^+)$ are of the same order.

\section{Proof of Theorem~\ref{thm:long_cycle} and Theorem~\ref{thm:rconnected_corollary}. }

In this section we explain how to deduce both Theorems~\ref{thm:long_cycle} and~\ref{thm:rconnected_corollary} from Lemma~\ref{lem:separator_size}. We start by recording the following folklore-observation about intersections of longest paths and cycles in graphs, which we will need repeatedly, see for example~\cite{zamfirescuagain} for a reference.
\begin{obs}\label{obs:intersect}
Let $G$ be a connected graph. Then any two longest paths in $G$ meet in some vertex. Additionally, if $G$ is $2$-connected, then any two longest cycles in $G$ meet in some vertex.
\end{obs}
Let us now continue by first proving Theorem~\ref{thm:rconnected_corollary}.
\begin{proof}[Proof of Theorem~\ref{thm:rconnected_corollary}] Let $G$ be an $r$-connected graph, and let $C_1, C_2$ be two longest cycles in $G$. Since $r\ge 2$ by assumption, Observation~\ref{obs:intersect} implies that the cycles $C_1$ and $C_2$ must overlap in at least one vertex.
This implies that $k:=|V(C_1)\cap V(C_2)|>0$. So by Lemma~\ref{lem:separator_size}, a smallest vertex-separator $S$ between $V(C_1)\setminus V(C_2)$ and $V(C_2)\setminus V(C_1)$ in $G$ must have size at most $|S|\le O(k^{8/5})$. On the other hand, since $G$ is $r$-connected, we must have $|S|\ge r$, implying $k\ge \Omega(r^{5/8})$, as desired. 
\end{proof}

Let us now turn to the deduction of Theorem~\ref{thm:long_cycle}, which is slightly more elaborate. 
\begin{proof}[Proof of Theorem~\ref{thm:long_cycle}]
In the following, let us define a \emph{longest cycle transversal} as a set $S\subseteq V(G)$ which intersects every longest cycle of $G$. 
For $k\in \mathbb{N}$, let $s(k)$ denote the smallest integer $s$ such that every 2-connected graph $G$ containing two longest cycles $C_1$ and $C_2$ such that $|V (C_1)\cap V(C_2)| = k$ has a longest cycle transversal of size at most $s$.  DeVos~\cite[Lemma 3]{deVos} showed that $s(k)\leq k^2+k$, and the remainder of the proof in~\cite{deVos} directly shows that the length of the longest cycle in a connected vertex-transitive graph on at least $n\geq 3$ vertices is at least 
\begin{equation}
\label{eq:deVos}
\min_{k\in \mathbb{N}}\max\left\{\sqrt{kn},\frac{n}{s(k)}\right\}.
\end{equation}
We next show how to obtain the improved upper bound $s(k)=O(k^{8/5})$ from Lemma~\ref{lem:separator_size}. Given two longest cycles $C_1$ and $C_2$ in a $2$-connected graph $G$ with $|V(C_1)\cap V(C_2)|=k$, by Lemma~\ref{lem:separator_size} there exists a vertex-separator $S$ of size at most $O(k^{8/5})$ between $V(C_1)\setminus V(C_2)$ and $V(C_2)\setminus V(C_1)$ in $G$. We now claim that $S\cup (V(C_1)\cap V(C_2))$ (which is of size $\le |S|+k=O(k^{8/5})$) intersects all longest cycles in $G$. Indeed, for any longest cycle $C$ in $G$, since $G$ is $2$-connected and since by Observation~\ref{obs:intersect} any two longest cycles in a $2$-connected graph must intersect each other, we must have $V(C)\cap V(C_i)\neq \emptyset$ for $i=1,2$. Hence, either $C$ contains a vertex in $V(C_1) \cap V(C_2)$ or a segment of $C$ forms a path from $V(C_1)\setminus V(C_2)$ to $V(C_2)\setminus V(C_1)$ in $G$, in which case this segment (and thus $C$) must be hit by $S$. This shows that indeed, $s(k)\le O(k^{8/5})$.

Plugging our improved bound $s(k)=O(k^{8/5})$ into (\ref{eq:deVos}) yields Theorem \ref{thm:long_cycle}. Indeed, let $C>0$ be an absolute constant such that $s(k)\le Ck^{8/5}$ for all $k\ge 1$. The minimum of the expression $$\max\left\{\sqrt{kn},\frac{n}{Ck^{8/5}}\right\}$$ among real values $k>0$ is achieved for a value of $k$ for which the two bounds $\sqrt{kn}$ and $n/(Ck^{8/5})$ are equal, and setting
\[
\sqrt{kn}=\frac{n}{Ck^{8/5}}
\]
yields
\[
k=n^{(1-1/2)/(1/2+8/5)}/C^{1/(1/2+8/5)}=\Theta(n^{\frac{5}{21}}).
\]
Plugging this into the above then gives the desired lower bound of $\sqrt{kn}=\Theta(n^{1/2+5/42})=\Theta(n^{13/21})$ on the length of a longest cycle. This proves Theorem~\ref{thm:long_cycle}.
\end{proof}

\section{Reduction to computer search}\label{sec:defs}
In this section, we introduce definitions that are required to formalize our computer-based approach and prove \cref{lem:weightlongcycle}, which will be used to prove Lemma \ref{lem:separator_size}. Of course, we cannot simply consider all possible graphs in our computer search, and rather do a brute-force search over the possible ``skeletons'' that may arise.

While reading the definitions in this section, the reader may wish to keep the following construction of an auxiliary bipartite graph in mind as motivation. Given a graph $G$ with two longest cycles $C_1$ and $C_2$ and a collection of vertex-disjoint paths $\mathcal{P}$ from $C_1 - V(C_2)$ to $C_2 - V(C_1)$, we may define a bipartite graph as follows. The vertices are the connected components of $C_1 -V(C_2)$ and $C_2 - V(C_1)$, and there is an edge between two such components if there is a vertex-disjoint path $P$ between them with $P\in \mathcal{P}$.

\medskip

An \emph{ordered bipartite graph} is a triple $\mathcal H=(H,\sigma_X,\sigma_Y)$, where
\begin{itemize}
    \item $H=(X,Y,E)$ is a bipartite graph,
    \item $\sigma_X$ is a total order on $X$, and
    \item $\sigma_Y$ is an total order on $Y$.
\end{itemize}

Given an ordered bipartite graph $\mathcal H=(H,\sigma_X,\sigma_Y)$, an \emph{$\mathcal H$-configuration} is a tuple $\tau=(\tau_h)_{h\in V(H)}$, where $\tau_h$ is a total order on the neighbours of $h$ for each $h\in V(H)$. 
Associated with this configuration, we create what we call a \emph{configuration graph}, denoted by $G_\tau$, defined as follows. First we create a vertex $v_{x,y}$ for each $x\in X$ and $y\in N_H(x)$ and let $C_1^\tau :=\{v_{x,y} : x\in X, ~ y\in N_{H}(x)\}$. We turn this into a cycle using the orders $(\sigma_X,(\tau_x)_{x\in X})$ ``lexicographically''. That is, we consider the total order $\leq_{C_1^\tau}$ obtained by 
\begin{align*}
    v_{x_1,y_1}\leq_{C_1^\tau} v_{x_2,y_2} & \quad \text{if $x_1<_{\sigma_X} x_2$},\\
     v_{x_1,y_1}\leq_{C_1^\tau} v_{x_1,y_2} & \quad \text{if $y_1\leq_{\tau_{x_1}} y_2$}.
\end{align*}
We place edges to turn $C_1^\tau$ into a cycle\footnote{If $C_1^\tau$ has only one vertex, this yields a loop in the graph, and if $C_1^\tau$ has two vertices, this creates parallel edges.}, adding an edge between the smallest and largest vertex in this order and otherwise adding edges between vertices that are consecutive in the total order.
Although $C_1^\tau$ is undirected, we also fix a cyclic order on it, which arises naturally by following the order described above. 
We define $C_2^\tau$ analogously by switching the roles of $X$ and $Y$. Finally, we define $G_\tau$ as the union of $C_1^\tau$ and $C_2^\tau$ together with the edges $\{\{v_{x,y},v_{y,x}\} : \{x,y\} \in E(H)\}\}$. We will call these ``edges corresponding to $H$'' the \emph{cross edges}.
We call $C_1^\tau$ and $C_2^\tau$ the \emph{main cycles} of $G_\tau$.

\medskip

Roughly speaking, our goal in this section is to show that when $G_\tau$ is a configuration graph for $\tau$ an $(H,\sigma_X,\sigma_Y)$-configuration, where $H\in \{Q_3^{+},K_{3,3}\}$, then there exists no assignment of ``lengths'' to its edges such that $C_1^{\tau}$ and $C_2^{\tau}$ are both longest cycles. In fact, provided that $C_1^\tau$ and $C_2^\tau$ have the same length we will always be able to find a longer cycle with special properties. The latter will allow us to also find a longer cycle in the setting of \cref{lem:separator_size} in our next section.

Given a configuration graph $G_\tau$, we say an edge $\{v_{x_1,y_1},v_{x_2,y_2}\}$ of $C_1^\tau$ is \emph{dangerous} if $x_1\neq x_2$, and analogously for the edges of $C_2^\tau$. Edges of $C_1^\tau$ and $C_2^\tau$ are otherwise said to be \emph{safe}. We say a cycle in $G_\tau$ is \emph{good} if it has no dangerous edges from $C_1^\tau$ or has no dangerous edges from $C_2^\tau$ (or from neither).

We also need an appropriate formal definition for the length of cycles in configuration graphs. Given a configuration graph $G_\tau$, we say a \emph{weight function} on it is a function $w:E(G_\tau)\rightarrow \mathbb{R}_{\geq 0}$ such that $w(e)\geq 1$ for every cross edge $e$, and  such that $w(C_1^\tau)=w(C_2^\tau)$, where we define the weight of any subgraph $J$ of $G_\tau$ by $w(J):=\sum_{e\in E(J)}w(e)$. We say a cycle $C$ in a weighted configuration graph $(G_\tau,w)$ is \emph{long} if its weight is strictly greater than that of the main cycles of $G_\tau$, that is if $w(C)>w(C_1)=w(C_2)$.

The next lemma will be useful for checking for the existence of good long cycles independent of the weight function.

\begin{lem}\label{lem:pathcheck}
    Let $\mathcal H=(H,\sigma_X,\sigma_Y)$ be an ordered bipartite graph and let $\tau$ be an $\mathcal H$-configuration. If there are two good cycles $C$ and $C'$ in $G_\tau$ which together contain all of the edges of its main cycles, as well as at least one cross edge, then for any weight function $w$, $(G_\tau,w)$ contains a good long cycle.
\end{lem}

\begin{proof}
    As $C$ and $C'$ together contain all edges of the main cycles, and at least one cross edge $e$ (which has non-zero weight), $$w(C)+w(C')\geq w(C_1)+w(C_2)+w(e)>2 w(C_1),$$
    and so $w(C)>w(C_1)$ or $w(C')>w(C_1)$.
\end{proof}

Given an ordered bipartite graph $\mathcal H=(H,\sigma_X,\sigma_Y)$ and an $\mathcal H$-configuration $\tau$, we define a linear program LP$_{\mathcal{H},\tau}$ as follows.
We create a variable $w_e\geq 0$ for each edge $e$ of $G_\tau$ and put the constraint $w_e\geq 1$ for all cross edges $e$ of $G_\tau$. We also have the following constraint that both main cycles must have the same weight:
\[\sum_{e\in E(C_1^\tau)}w_e= \sum_{e\in E(C_2^\tau)}w_e.
\]
Finally, for every good cycle $C$ in $G_\tau$, we create the following constraint:
\[\sum_{e\in E(C)}w_e\leq \sum_{e\in E(C_1^\tau)}w_e.
\]

We are only interested in the feasibility, so we may set the objective function to 1. This linear program is not feasible (there are no solutions) if and only if for every weight function $w$, $(G_\tau,w)$ contains a good long cycle. We record this fact below.
\begin{obs}\label{obs:linearprogram}
    Let $\mathcal H=(H,\sigma_X,\sigma_Y)$ be an ordered bipartite graph and let $\tau$ be an $\mathcal H$-configuration. Then LP$_{\mathcal{H},\tau}$ is infeasible if and only if for all weight functions $w$, $(G_\tau,w)$ contains a good long cycle.
\end{obs}
\begin{proof}
By definition of LP$_{\mathcal{H},\tau}$ and by definition of a weight function, an edge-indexed vector $(w_e)_{e\in E(G_\tau)}$ is feasible for LP$_{\mathcal{H},\tau}$ if and only if (1) the corresponding mapping $w:E(G_\tau)\rightarrow \mathbb{R}$, defined as $w(e):=w_e$ for every $e$, is a weight function, and (2) we have $w(C)=\sum_{e\in E(C)}w_e\le \sum_{e\in E(C_1^\tau)}w_e=w(C_1^\tau)$ for every good cycle $C$. The latter is equivalent to saying that there exist no good long cycles. Hence, it follows that LP$_{\mathcal{H},\tau}$ is feasible if and only if there exists a weight function $w$ such that $(G_\tau,w)$ contains no good long cycle. The claim of the lemma now follows by negating bot sides of this equivalence.
\end{proof}
Given that there are many possible configurations to consider, the following lemma will allow us to significantly reduce computation time. A reader which is not interested in the computational efficiency can safely skip over the concepts of substructures that we define below and consider performing a brute-force search instead.

Given two ordered bipartite graphs $\mathcal H=(H,\sigma_X,\sigma_Y)$ and $\mathcal H'=(H',\sigma_X',\sigma_Y')$, we say that $\mathcal H'$ is a \emph{subgraph} of $\mathcal H$ if $H'$ is a subgraph of $H$ and we have the following conditions on the restrictions of the orders: $\sigma_X|_{X'}=\sigma_X'$ and $\sigma_Y|_{Y'}=\sigma_Y'$.

Furthermore, if $\mathcal H'$ is a subgraph of $\mathcal H$, $\tau$ is an $\mathcal H$-configuration, and $\tau'$ is an $\mathcal H'$-configuration, then we say that $\tau'$ is a \emph{subconfiguration} of $\tau$ if $\tau'_h=\tau_h|_{V(H')}$ for every $h\in V(H')$.

Finally, let $\tau'$ be a subconfiguration of $\tau$, $w$ be a weight function on $G_\tau$ and $w'$ be a weight function on $G_{\tau'}$.  We say that $w$ \emph{extends} $w'$ if these weight functions are compatible with this relation, that is if they respect the following two conditions.
\begin{itemize}
    \item If $e$ is a cross edge of $G_{\tau'}$ (which is necessarily also a cross edge of $G_\tau$), then $w'(e)=w(e)$.
    \item Let $i\in\{1,2\}$ and suppose that $v,w\in V(C_i^{\tau'})\subseteq V(C_i^{\tau})$ satisfy $vw\in E(C_i^{\tau'})$. If $P$ is the unique path\footnote{When $C_1^{\tau'}$ has only two vertices $x$ and $x'$, we distinguish ``the edge from $x$ to $x'$'' and ``the edge from $x'$ to $x$'' to obtain a unique path here. When  $C_1^{\tau'}$ contains only one vertex $x$, this path refers to the loop on $x$.} 
    between these vertices in $C_i^\tau$ that contains no other vertices of $C_i^{\tau'}$, then $w(P)=w'(vw)$.
\end{itemize}
We are now ready to show that it is sufficient to find good long cycles in subconfigurations.
\begin{lem}\label{lem:subconfig}
Let $\mathcal H'$ and $\mathcal H$ denote ordered subgraphs, $\tau'$ an $\mathcal H'$-configuration and $\tau$ an $\mathcal H$-configuration. Suppose that $\tau'$ is a subconfiguration of $\tau$ and $w$ is a weight function on the configuration graph $G_\tau$ which extends a weight function $w'$ on the configuration graph $G_{\tau'}$. 
If $(G_{\tau'},w')$ contains a good long cycle, then $(G_\tau,w)$ also contains a good long cycle.
\end{lem}
\begin{proof}
    Any vertex of $G_{\tau'}$ has a corresponding vertex in $G_{\tau}$. When there is an edge between two vertices of $G_{\tau'}$ for which there is no edge in $G_\tau$, instead there will be a path following one of the main cycles in $G_\tau$. This gives a corresponding cycle in $G_\tau$ for each cycle in $G_{\tau'}$. By the definition of the extension of weight functions, these cycles always have the same weight. Furthermore, if a cycle is good in $G_{\tau'}$, then its corresponding cycle is good in $G_{\tau}$. This is because whether edges (that are not cross edges) are safe, depends on which vertices in the cycles they connect, and safe edges in $G_{\tau'}$ correspond to (paths consisting of) safe edges in $G_{\tau}$.  
    This proves the lemma.
\end{proof}

We note that there is another issue that we will need to address to reduce the computation time of our computer search. For an arbitrary bipartite graph $H$ with bipartition ($X,Y$), there could be up to $|X|!|Y|!$ possible ordered bipartite graphs. Luckily, symmetry can be used for the graphs of interest to us in order to reduce this number. 
We will detail the more general criterion in the proof of the main result of this section, which we now proceed with.

\begin{figure}
     \centering
     \captionsetup[subfigure]{justification=centering}
     \begin{subfigure}[t]{0.32\textwidth}
         \centering
         \begin{tikzpicture}[scale=1,
			dot/.style = {circle, fill, minimum size=#1,
			inner sep=0pt, outer sep=0pt},
			dot/.default = 4pt]

            \node[dot,label={left:$a_1$}] (A1) at (0,0) {};
            \node[dot,label={left:$a_2$}] (A2) at (0,1) {};
            \node[dot,label={left:$a_3$}] (A3) at (0,2) {};
            \node[dot,label={right:$b_1$}] (B1) at (2,0) {};
            \node[dot,label={right:$b_2$}] (B2) at (2,1) {};
            \node[dot,label={right:$b_3$}] (B3) at (2,2) {};
   
			\draw (A1) to (B1);
            \draw (A1) to (B2);
            \draw (A1) to (B3);
            \draw (A2) to (B1);
            \draw (A2) to (B2);
            \draw (A2) to (B3);
            \draw (A3) to (B1);
            \draw (A3) to (B2);
            \draw (A3) to (B3);
		\end{tikzpicture}
         \caption{Only ordered bipartite graph for $H=K_{3,3}$.}
     \end{subfigure}
     \hfill
     \begin{subfigure}[t]{0.33\textwidth}
         \centering
         \begin{tikzpicture}[scale=1,
			dot/.style = {circle, fill, minimum size=#1,
			inner sep=0pt, outer sep=0pt},
			dot/.default = 4pt]

            \node[dot,label={left:$a_1$}] (A1) at (0,0) {};
            \node[dot,label={left:$a_2$}] (A2) at (0,1) {};
            \node[dot,label={left:$a_3$}] (A3) at (0,2) {};
            \node[dot,label={left:$a_4$}] (A4) at (0,3) {};
            \node[dot,label={right:$b_1$}] (B1) at (2,0) {};
            \node[dot,label={right:$b_2$}] (B2) at (2,1) {};
            \node[dot,label={right:$b_3$}] (B3) at (2,2) {};
            \node[dot,label={right:$b_4$}] (B4) at (2,3) {};
   
			\draw (A1) to (B2);
            \draw (A1) to (B3);
            \draw (A1) to (B4);
            \draw (A2) to (B1);
            \draw (A2) to (B2);
            \draw (A2) to (B3);
            \draw (A2) to (B4);
            \draw (A3) to (B1);
            \draw (A3) to (B2);
            \draw (A3) to (B3);
            \draw (A4) to (B1);
            \draw (A4) to (B2);
            \draw (A4) to (B4);
		\end{tikzpicture}
         \caption{First ordered bipartite graph with $H=Q_3^+$.}
     \end{subfigure}
     \hfill
     \begin{subfigure}[t]{0.33\textwidth}
         \centering
         \begin{tikzpicture}[scale=1,
			dot/.style = {circle, fill, minimum size=#1,
			inner sep=0pt, outer sep=0pt},
			dot/.default = 4pt]

            \node[dot,label={left:$a_1$}] (A1) at (0,0) {};
            \node[dot,label={left:$a_2$}] (A2) at (0,1) {};
            \node[dot,label={left:$a_3$}] (A3) at (0,2) {};
            \node[dot,label={left:$a_4$}] (A4) at (0,3) {};
            \node[dot,label={right:$b_1$}] (B1) at (2,0) {};
            \node[dot,label={right:$b_2$}] (B2) at (2,1) {};
            \node[dot,label={right:$b_4$}] (B4) at (2,2) {};
            \node[dot,label={right:$b_3$}] (B3) at (2,3) {};
   
			\draw (A1) to (B2);
            \draw (A1) to (B3);
            \draw (A1) to (B4);
            \draw (A2) to (B1);
            \draw (A2) to (B2);
            \draw (A2) to (B3);
            \draw (A2) to (B4);
            \draw (A3) to (B1);
            \draw (A3) to (B2);
            \draw (A3) to (B3);
            \draw (A4) to (B1);
            \draw (A4) to (B2);
            \draw (A4) to (B4);
		\end{tikzpicture}
         \caption{Second ordered bipartite graph with $H=Q_3^+$.}
     \end{subfigure}
     
        \caption{The three ordered bipartite graphs $(H,\sigma_X,\sigma_Y)$ that need to be considered up to equivalence, as defined in \cref{lem:weightlongcycle}. The permutations $\sigma_X,\sigma_Y$ are implicit in the drawings.}
        \label{fig:configurations}
\end{figure}
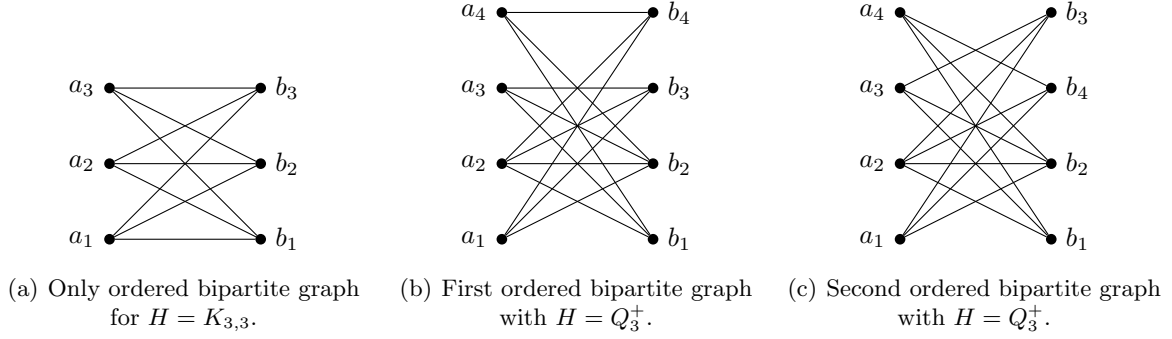

\begin{lem}\label{lem:weightlongcycle}
    Let $\mathcal H=(H,\sigma_X,\sigma_Y)$ be an ordered bipartite graph such that either $Q_3^+$ or $K_{3,3}$ is a subgraph of $H$, and let $\tau$ be an $\mathcal H$-configuration. For any weight function $w$ on $G_\tau$, $(G_\tau,w)$ contains a good long cycle.
\end{lem}
\begin{proof}
    By applying \cref{lem:subconfig}, we may assume that $H\in \{Q_3^+,K_{3,3}\}$.

    In our code, we have first implemented checking for equivalence between ordered bipartite graphs. We generate all possible ordered bipartite graphs with $H\in \{Q_3^+,K_{3,3}\}$, and then test for equivalence: there are up to equivalence only two ordered bipartite graphs for $Q_3^+$, and of course only one for $K_{3,3}$, as depicted in \cref{fig:configurations}. 
    The equivalence is checked by verifying whether the ordered bipartite graphs admit the same collection of configurations graphs (up to graph isomorphism). For an ordered bipartite graph $\mathcal H=(H,\sigma_X,\sigma_Y)$, we place a cycle on the vertices in $X$ and a cycle on the vertices in $Y$ in the cyclic orders given by $\sigma_X$ and $\sigma_Y$ respectively (adding also the edge between the last and first vertex). 
    If, between two such auxiliary graphs, there exists an isomorphism such that $X$ gets mapped to $X$ or to $Y$, then the set of configuration graphs remains the same (up to relabelling). This follows directly from the definition of the configuration graphs.

    We then run a computer search to check that for every configuration $\tau$ on these three ordered bipartite graphs, every weighted configuration graph resulting from $\tau$ contains a good long cycle. As there are many configurations even on these three ordered bipartite graphs, we proceed by adding edges iteratively. Indeed, it is a direct consequence of \cref{lem:subconfig} that is it sufficient to check the condition for at least one subconfiguration of every possible configuration.
    \begin{enumerate}
        \item We fix an order on the edges of $H$. Let $\emptyset=H_0\subset H_1\subset\dots\subset H_{|E(H)|}=H$ be a sequence of subgraphs in which one edge is added at every step\footnote{Although our implementation does not optimize this order, it appears one can make the computation quicker by selecting an order on the edges of $H$ which prunes the most configurations earliest in the process. In particular, as is intuitive, maximizing the number of cycles appearing in $H_i$ in the first steps generally reduces the number of configurations that appear in later steps, with some caveats.}. We will write $e_i$ for the unique edge in $E(H_i)\setminus E(H_{i-1})$.
        \item At step $i=0$, let $S_0$ be the set containing only the empty configuration.
        \item At step $i\in \{1,\dots,|E(H)|\}$, we construct $S_i$ as follows. For each $(H_{i-1},\sigma_X,\sigma_Y)$-configuration $\tau'\in S_{i-1}$, consider every  $(H_{i},\sigma_X,\sigma_Y)$-configuration $\tau$ such that $\tau'$ is a sub-configuration of $\tau$. Let $R_i$ be the set of these configurations. Add $\tau$ to $S_i$ if the criterion of \cref{lem:pathcheck} cannot guarantee that every weighted configuration graph on $\tau$ would contain a good long cycle.
        \item Finally, we create a set $S$ as follows. For every $H$-configuration $\tau$ in $S_{|E(H)|}$, we add $\tau$ to $S$ if the more powerful linear program criterion from \cref{obs:linearprogram} cannot guarantee that every weighted configuration graph on $\tau$ would contain a good long cycle.
    \end{enumerate}

    If $S$ is empty, the theorem is a direct consequence of \cref{lem:subconfig,lem:pathcheck,obs:linearprogram}.

    This is indeed the case. We have implemented this algorithm in Mathematica \cite{wolfram_research_inc_mathematica_nodate}. The code is available at \cite{code}, and is thoroughly commented. A summary of results per step (number of configurations in $R_i$, $S_i$) is presented in \cref{tab:computationresults}. The computation took under 1h49m\footnote{Of course, the computation time could be greatly reduced through various optimizations, for instance by implementing the code in a lower level language or by using better pruning heuristics.} on a 2020 MacBook Air with M1 chip and 16 GB ram running Mathematica 13.0.0.0.
\end{proof}

\newcolumntype{?}{!{\vrule width 1pt}}
\renewcommand{\arraystretch}{1.2} 

\begin{table}
\captionsetup[subtable]{justification=centering}
\begin{subtable}{0.32\textwidth}
    \begin{tabular}{llll}
\Xcline{1-4}{1pt}
\multicolumn{1}{?c|}{$i$ } & \multicolumn{1}{c|}{$e_i$} & \multicolumn{1}{c|}{$|R_i|$}    & \multicolumn{1}{c?}{$|S_i|$}    \\ \Xcline{1-4}{1pt}
\multicolumn{1}{?c|}{$1$} & \multicolumn{1}{c|}{$\{a_1,b_1\}$} & \multicolumn{1}{c|}{$1$}    & \multicolumn{1}{c?}{$1$}   \\ \hline
\multicolumn{1}{?c|}{$2$} & \multicolumn{1}{c|}{$\{a_1,b_2\}$} & \multicolumn{1}{c|}{$2$}    & \multicolumn{1}{c?}{$2$}   \\ \hline
\multicolumn{1}{?c|}{$3$} & \multicolumn{1}{c|}{$\{a_1,b_3\}$} & \multicolumn{1}{c|}{$6$}    & \multicolumn{1}{c?}{$6$}   \\ \hline
\multicolumn{1}{?c|}{$4$} & \multicolumn{1}{c|}{$\{a_2,b_1\}$} & \multicolumn{1}{c|}{$12$}   & \multicolumn{1}{c?}{$12$}  \\ \hline
\multicolumn{1}{?c|}{$5$} & \multicolumn{1}{c|}{$\{a_2,b_2\}$} & \multicolumn{1}{c|}{$48$}   & \multicolumn{1}{c?}{$36$}  \\ \hline
\multicolumn{1}{?c|}{$6$} & \multicolumn{1}{c|}{$\{a_2,b_3\}$} & \multicolumn{1}{c|}{$216$}  & \multicolumn{1}{c?}{$120$} \\ \hline
\multicolumn{1}{?c|}{$7$} & \multicolumn{1}{c|}{$\{a_3,b_1\}$} & \multicolumn{1}{c|}{$360$}  & \multicolumn{1}{c?}{$360$} \\ \hline
\multicolumn{1}{?c|}{$8$} & \multicolumn{1}{c|}{$\{a_3,b_2\}$} & \multicolumn{1}{c|}{$2160$} & \multicolumn{1}{c?}{$684$} \\ \hline
\multicolumn{1}{?c|}{$9$} & \multicolumn{1}{c|}{$\{a_3,b_3\}$} & \multicolumn{1}{c|}{$6156$} & \multicolumn{1}{c?}{$72$}  \\ \Xcline{1-4}{1pt}
                          &              & \multicolumn{1}{?c|}{$|S|$}    & \multicolumn{1}{c?}{0}     \\ \Xcline{3-4}{1pt}
\end{tabular}
\caption{Results for the ordered bipartite graph with $H=K_{3,3}$.}
\end{subtable}
\begin{subtable}{0.32\textwidth}
    \begin{tabular}{llll}
\Xcline{1-4}{1pt}
\multicolumn{1}{?c|}{$i$ } & \multicolumn{1}{c|}{$e_i$} & \multicolumn{1}{c|}{$|R_i|$}    & \multicolumn{1}{c?}{$|S_i|$}    \\ \Xcline{1-4}{1pt}
\multicolumn{1}{?c|}{$1$} & \multicolumn{1}{c|}{$\{a_1,b_2\}$} & \multicolumn{1}{c|}{$1$}    & \multicolumn{1}{c?}{$1$}   \\ \hline
\multicolumn{1}{?c|}{$2$} & \multicolumn{1}{c|}{$\{a_1,b_3\}$} & \multicolumn{1}{c|}{$2$}    & \multicolumn{1}{c?}{$2$}   \\ \hline
\multicolumn{1}{?c|}{$3$} & \multicolumn{1}{c|}{$\{a_2,b_2\}$} & \multicolumn{1}{c|}{$4$}    & \multicolumn{1}{c?}{$4$}   \\ \hline
\multicolumn{1}{?c|}{$4$} & \multicolumn{1}{c|}{$\{a_2,b_3\}$} & \multicolumn{1}{c|}{$16$}   & \multicolumn{1}{c?}{$12$}  \\ \hline
\multicolumn{1}{?c|}{$5$} & \multicolumn{1}{c|}{$\{a_3,b_2\}$} & \multicolumn{1}{c|}{$36$}   & \multicolumn{1}{c?}{$36$}  \\ \hline
\multicolumn{1}{?c|}{$6$} & \multicolumn{1}{c|}{$\{a_3,b_3\}$} & \multicolumn{1}{c|}{$216$}  & \multicolumn{1}{c?}{$120$} \\ \hline
\multicolumn{1}{?c|}{$7$} & \multicolumn{1}{c|}{$\{a_1,b_4\}$} & \multicolumn{1}{c|}{$360$}  & \multicolumn{1}{c?}{$360$} \\ \hline
\multicolumn{1}{?c|}{$8$} & \multicolumn{1}{c|}{$\{a_2,b_4\}$} & \multicolumn{1}{c|}{$2160$} & \multicolumn{1}{c?}{$684$} \\ \hline
\multicolumn{1}{?c|}{$9$} & \multicolumn{1}{c|}{$\{a_4,b_4\}$} & \multicolumn{1}{c|}{$2052$} & \multicolumn{1}{c?}{$2052$}  \\
\hline
\multicolumn{1}{?c|}{$10$} & \multicolumn{1}{c|}{$\{a_2,b_1\}$} & \multicolumn{1}{c|}{$8208$} & \multicolumn{1}{c?}{$8208$}  \\
\hline
\multicolumn{1}{?c|}{$11$} & \multicolumn{1}{c|}{$\{a_3,b_1\}$} & \multicolumn{1}{c|}{$49248$} & \multicolumn{1}{c?}{$8838$}  \\
\hline
\multicolumn{1}{?c|}{$12$} & \multicolumn{1}{c|}{$\{a_4,b_1\}$} & \multicolumn{1}{c|}{$53028$} & \multicolumn{1}{c?}{$7492$}  \\
\hline
\multicolumn{1}{?c|}{$13$} & \multicolumn{1}{c|}{$\{a_4,b_2\}$} & \multicolumn{1}{c|}{$89904$} & \multicolumn{1}{c?}{$12$}  \\\Xcline{1-4}{1pt}
                          &              & \multicolumn{1}{?c|}{$|S|$}    & \multicolumn{1}{c?}{0}     \\ \Xcline{3-4}{1pt}
\end{tabular}
\caption{Results for the first ordered bipartite graph with $H=Q_3^+$.}
\end{subtable}
\begin{subtable}{0.32\textwidth}
    \begin{tabular}{llll}
\Xcline{1-4}{1pt}
\multicolumn{1}{?c|}{$i$ } & \multicolumn{1}{c|}{$e_i$} & \multicolumn{1}{c|}{$|R_i|$}    & \multicolumn{1}{c?}{$|S_i|$}    \\ \Xcline{1-4}{1pt}
\multicolumn{1}{?c|}{$1$} & \multicolumn{1}{c|}{$\{a_1,b_2\}$} & \multicolumn{1}{c|}{$1$}    & \multicolumn{1}{c?}{$1$}   \\ \hline
\multicolumn{1}{?c|}{$2$} & \multicolumn{1}{c|}{$\{a_1,b_3\}$} & \multicolumn{1}{c|}{$2$}    & \multicolumn{1}{c?}{$2$}   \\ \hline
\multicolumn{1}{?c|}{$3$} & \multicolumn{1}{c|}{$\{a_2,b_2\}$} & \multicolumn{1}{c|}{$4$}    & \multicolumn{1}{c?}{$4$}   \\ \hline
\multicolumn{1}{?c|}{$4$} & \multicolumn{1}{c|}{$\{a_2,b_3\}$} & \multicolumn{1}{c|}{$16$}   & \multicolumn{1}{c?}{$12$}  \\ \hline
\multicolumn{1}{?c|}{$5$} & \multicolumn{1}{c|}{$\{a_3,b_2\}$} & \multicolumn{1}{c|}{$36$}   & \multicolumn{1}{c?}{$36$}  \\ \hline
\multicolumn{1}{?c|}{$6$} & \multicolumn{1}{c|}{$\{a_3,b_3\}$} & \multicolumn{1}{c|}{$216$}  & \multicolumn{1}{c?}{$120$} \\ \hline
\multicolumn{1}{?c|}{$7$} & \multicolumn{1}{c|}{$\{a_1,b_4\}$} & \multicolumn{1}{c|}{$360$}  & \multicolumn{1}{c?}{$360$} \\ \hline
\multicolumn{1}{?c|}{$8$} & \multicolumn{1}{c|}{$\{a_2,b_4\}$} & \multicolumn{1}{c|}{$2160$} & \multicolumn{1}{c?}{$684$} \\ \hline
\multicolumn{1}{?c|}{$9$} & \multicolumn{1}{c|}{$\{a_4,b_4\}$} & \multicolumn{1}{c|}{$2052$} & \multicolumn{1}{c?}{$2052$}  \\
\hline
\multicolumn{1}{?c|}{$10$} & \multicolumn{1}{c|}{$\{a_2,b_1\}$} & \multicolumn{1}{c|}{$8208$} & \multicolumn{1}{c?}{$8208$}  \\
\hline
\multicolumn{1}{?c|}{$11$} & \multicolumn{1}{c|}{$\{a_3,b_1\}$} & \multicolumn{1}{c|}{$49248$} & \multicolumn{1}{c?}{$11586$}  \\
\hline
\multicolumn{1}{?c|}{$12$} & \multicolumn{1}{c|}{$\{a_4,b_1\}$} & \multicolumn{1}{c|}{$69516$} & \multicolumn{1}{c?}{$508$}  \\
\hline
\multicolumn{1}{?c|}{$13$} & \multicolumn{1}{c|}{$\{a_4,b_2\}$} & \multicolumn{1}{c|}{$6096$} & \multicolumn{1}{c?}{$0$}  \\\Xcline{1-4}{1pt}
                          &              & \multicolumn{1}{?c|}{$|S|$}    & \multicolumn{1}{c?}{0}     \\ \Xcline{3-4}{1pt}
\end{tabular}
\caption{Results for the second ordered bipartite graph with $H=Q_3^+$.}
\end{subtable}

\caption{Summary of results per step for the computer-search algorithm presented in the proof of \cref{lem:weightlongcycle}. Refer to \cref{fig:configurations} for the vertex labels.}
\label{tab:computationresults}
\end{table}

\section{Proof of \cref{lem:separator_size}}\label{sec:lemmaproof}

We will now deduce \cref{lem:separator_size} from \cref{lem:weightlongcycle}. 

\begin{proof}[Proof of \cref{lem:separator_size}]
Let $C_1$ and $C_2$ be two longest cycles in a graph $G$ such that $|V(C_1) \cap V(C_2)| = k>0$. Let $S$ be a minimum vertex separator between $A = V(C_1) \setminus V(C_2)$ and $B=V(C_2)\setminus V(C_1)$ in~$G$. Suppose towards a contradiction that $|S| \geq  \ex(2k, \{Q_3^+, K_{3, 3}\})+1$. 

By Menger's Theorem, there are at least $\ex(2k, \{Q_3^+, K_{3, 3}\}) + 1$ vertex-disjoint paths with one endpoint in $A$ and one endpoint in $B$. Fix such a collection $\mathcal{P}$. We also fix total orderings on $C_1$ and $C_2$ by choosing a smallest element and a direction.

We define an ordered bipartite graph as follows.
Let $X$ be the set of connected components of $C_1 - V(C_2)$ and let $Y$ be the set of connected components of $C_2 - V(C_1)$. The sets $X,Y$ inherit a total order from  $C_1,C_2$ respectively.

We add an edge between $x\in X$ and $y\in Y$ if there is some $P \in \mathcal{P}$ with one endpoint of $P$ in the connected component $x$ and the other endpoint in $y$. 

For each edge $xy$ with $x\in X$ and $y\in Y$, there is a unique such path that we will denote by $P_{xy}$. Indeed, suppose on the contrary there were two such paths $P_{xy}$ and $P_{xy}'$. Let $v$ and $v'$ be the endpoints of $P_{xy}$ and $P_{xy}'$ respectively in $x$ and let $u, u'$ be the endpoints of $P_{xy}, P_{xy}'$ respectively in $y$. Let $Q$ be the path between $v$ and $v'$ in $x$ and $Q'$ be the path between $u$ and $u'$ in $y$. This gives two cycles 
\begin{align*}
    &Q \rightarrow P_{xy} \rightarrow (C_2 \setminus Q') \rightarrow P_{xy}'\rightarrow \text{ and }\\
    &Q'\rightarrow P_{xy} \rightarrow (C_1 \setminus Q)\rightarrow P_{xy}'\rightarrow
\end{align*}
which together cover all edges of $C_1\cup C_2$, plus at least two more. Hence one of the two cycles above is longer than $C_1$ and $C_2$, contradicting that they are longest cycles. Thus, there is a unique such path. 

This concludes the description of the ordered bipartite graph $\mathcal{H}=(X,Y,E)$. By what we have shown in the previous paragraph, we have $|E|=|\mathcal{P}|$.

We next define an $\mathcal{H}$-configuration $\tau$. 
Let $x \in X$. In the configuration graph $G_\tau$ of $\tau$, there will be $|N_H(x)|$ copies of $x$ in $C_1^\tau$, one per neighbour $y\in Y$ of $x$.
We order the neighbours of $x$ according to the order in which these ``edges'' (paths from $\mathcal{P}$) intersect the original cycle $C_1$. To be precise, we consider the paths $P_{u_1,v_1},\dots,P_{u_d,v_d}\in \mathcal{P}$ incident with the connected component of $C_1\setminus C_2$ corresponding to $x\in X$, with $u_1,\dots,u_d\in V(C_1)$ numbered in order of the total ordering we fixed on $C_1$. As we argued above, $v_1,\dots,v_d$ are vertices in distinct components $y_1,\dots,y_d\in Y$, and this is the order we place on the neighbours of $x$. 

\begin{figure}
\tikzset{every picture/.style={line width=0.75pt}} 

\includegraphics[scale=0.22]{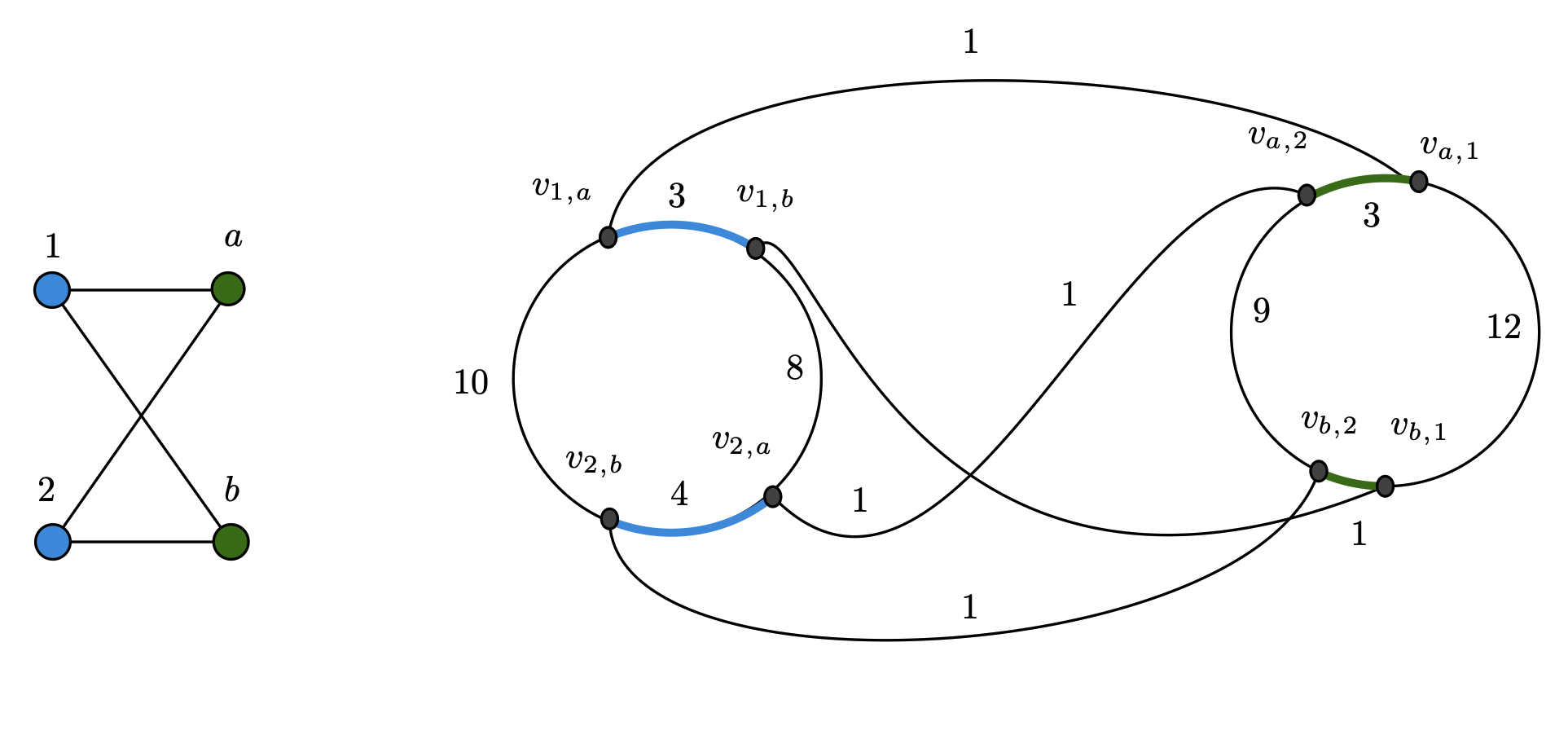}
\caption{The above depicts a bipartite graph $H$  (left) together with a possible configuration graph $G_\tau$ (right). The coloured edges are safe.}
\label{fig:configuration_example}
\end{figure}

We obtain a configuration graph $G_{\tau}$ with main cycles $C_1^{\tau}$ and $C_2^{\tau}$ as described in Section~\ref{sec:defs}. We included an example in Figure~\ref{fig:configuration_example}. Recall that $C_1^\tau$ contains vertices $v_{x,y}$ for $x\in X$ and $y\in N_H(x)$. 

We define a weight function $w: E(G_{\tau}) \rightarrow \R_{\geq0}$, considering the two types of edges separately.
\begin{itemize}
    \item If $e=v_{x,y}v_{y,x}$ is a cross edge with $x\in X$ and $y\in Y$, then we have defined a path $P_{xy}\in \mathcal{P}$. We set $w(e)$ as the length of $P_{xy}$.
    \item Otherwise, the edge is contained in one of the main cycles. Since both cases are analogous, we describe the case for an edge  $e=v_{x,y}v_{x',y'}$ of $C_1^\tau$ with $x,x'\in X$ and $y,y'\in Y$. Let $u_{xy},u_{x'y'}\in V(C_1)$ denote the endpoints in $C_1$  of $P_{xy},P_{x'y'}\in \mathcal{P}$ respectively. We set $w(e)$ as the length of the arc of $C_1$ from $u_{xy}$ to $u_{x'y'}$, in the direction we chose for the total ordering on $C_1$.
\end{itemize} 

Since $|E|=|\mathcal{P}|\geq \ex(2k, \{ Q_{3}^+, K_{3, 3}\})+1$, it follows that the bipartite graph $H=(X,Y,E)$ contains either $Q_3^+$ or $K_{3, 3}$. Thus, by \cref{lem:weightlongcycle}, $(\mathcal{G}_\tau,W)$ contains a good long cycle $C^\tau$. The vertices $v^\tau_1,\dots,v^\tau_\ell$ of $C^\tau$ correspond to vertices $v_1,\dots,v_\ell$ of $G$. The path which was used to define the edge $e=v_i^\tau v_{i+1}^\tau$ of $C^\tau$ has length $w(e)$ and is between $v_i$ and $v_{i+1}$. Therefore, we obtain a closed walk $C$ in $G$ of length $w(C^{\tau})$. Since $C^\tau$ is long, $w(C^\tau)>w(C_1^\tau)$ and the length of $C$ is strictly larger than the length of $C_1$. What remains to show is that $C$ is in fact a cycle. The vertices $v_1,\dots,v_\ell$ are distinct because the vertices $v^\tau_1,\dots,v^\tau_\ell$ are distinct, but we need to show the paths between them do not intersect internally. The paths corresponding to cross edges are chosen to be internally vertex-disjoint and only intersect $V(C_1)\cup V(C_2)$ in their endpoints. For $i\in \{1,2\}$, the paths corresponding to edges from $C_i^\tau$ correspond to internally disjoint arcs of $C_i$. It remains to rule out intersections between the arcs of $C$ on $C_1$ and $C_2$: By the definition of good, there exists some $i \in \{1,2\}$ such that all arcs of $C_i$ that are used by $C$ are part of a connected component of $C_i-V(C_{3-i})$ and thus do not intersect any of the arcs that $C$ uses from $C_{3-i}$. This shows $C$ is indeed a cycle, contradicting $C_1$ being the longest cycle.

We conclude that $|S| \leq \ex(2k, \{ Q_{3}^+, K_{3, 3}\})$ as desired.
\end{proof}

\section{Further remarks}\label{sec:concludingremark}
In this paper, we presented progress on two old conjectures about longest cycles via Lemma~\ref{lem:separator_size}. Any further improvement on this lemma directly improves the bounds towards the conjectures, and in particular it is possible that our $O(k^{8/5})$ upper bound can be improved to $O(k)$. This would imply that every two longest cycles in an $r$-connected graph meet in at least $\Omega(r)$ vertices and that every connected vertex-transitive graph on $n\geq 3$ vertices has a cycle of length at least $\Omega(n^{2/3})$. 

One direction towards improving the bounds in Lemma~\ref{lem:separator_size} is to prove a result analogous to Lemma~\ref{lem:weightlongcycle} for a graph $H$ with a better extremal number than $Q_3^+$. In fact, Chen, Faudree and Gould~\cite{ChenFaudreeGould98} already took a similar approach back in 1998 and proved a result analogous to Lemma~\ref{lem:weightlongcycle} for $H=K_{3,257}$ without the aid of a computer, via a result analogous to our Lemma~\ref{lem:pathcheck}, the Erd\H{o}s-Szekeres theorem and the pigeonhole principle. The perhaps most natural attempt is to consider the complete bipartite graphs $K_{2,t}$, since their extremal functions satisfy $\ex(n,K_{2,t})=\Theta(n^{3/2})\ll n^{8/5}$. However, one can show that for every fixed integer $t\in \mathbb{N}$ there exist (weighted) configuration graphs on an ordered bipartite graph of $K_{2,t}$ without good long cycles. Similarly, by computer search, we found that Lemma~\ref{lem:weightlongcycle} is false when $Q_3^+$ is replaced by $Q_3$: there are (weighted) configuration graphs on an ordered bipartite graph of $Q_3$ without good long cycles. Indeed, these can be obtained with only slight modifications to the code in \cite{code}, see \cref{fig:hypercubebadexample} for an example. 

\begin{figure}
    \centering
    \begin{tikzpicture}[scale=1.3,
			dot/.style = {circle, fill, minimum size=#1,
			inner sep=0pt, outer sep=0pt},
			dot/.default = 4pt]

            \node[dot] (A1) at (0,0) {};
            \node[dot] (A2) at (1,0) {};
            \node[dot] (A3) at (2,0) {};
            \node[dot] (A4) at (3,0) {};
            \node[dot] (A5) at (4,0) {};
            \node[dot] (A6) at (5,0) {};
            \node[dot] (A7) at (6,0) {};
            \node[dot] (A8) at (7,0) {};
            \node[dot] (A9) at (8,0) {};
            \node[dot] (A10) at (9,0) {};
            \node[dot] (A11) at (10,0) {};
            \node[dot] (A12) at (11,0) {};

            \node[dot] (B1) at (0,-2) {};
            \node[dot] (B2) at (1,-2) {};
            \node[dot] (B3) at (2,-2) {};
            \node[dot] (B4) at (3,-2) {};
            \node[dot] (B5) at (4,-2)  {};
            \node[dot] (B6) at (5,-2) {};
            \node[dot] (B7) at (6,-2) {};
            \node[dot] (B8) at (7,-2) {};
            \node[dot] (B9) at (8,-2) {};
            \node[dot] (B10) at (9,-2) {};
            \node[dot] (B11) at (10,-2) {};
            \node[dot] (B12) at (11,-2) {};
   
			\draw (A1) to node[midway, above]{$21$} (A2);
            \draw (A2) to node[midway, above]{$1$} (A3);
            \draw[dashed] (A3) to node[midway, above]{$31$} (A4);
            \draw (A4) to node[midway, above]{$33$} (A5);
            \draw (A5) to node[midway, above]{$1$} (A6);
            \draw[dashed] (A6) to node[midway, above]{$70$} (A7);
            \draw (A7) to node[midway, above]{$1$} (A8);
            \draw (A8) to node[midway, above]{$3$} (A9);
            \draw[dashed] (A9) to node[midway, above]{$71$} (A10);
            \draw (A10) to node[midway, above]{$53$} (A11);
            \draw (A11) to node[midway, above]{$13$} (A12);
            \draw[dashed] (A12) to[out=110,in=70, looseness=0.3] node[below]{$38$} (A1);

            \draw (B1) to node[midway, below]{$29$} (B2);
            \draw (B2) to node[midway, below]{$39$} (B3);
            \draw[dashed] (B3) to node[midway, below]{$81$} (B4);
            \draw (B4) to node[midway, below]{$4$} (B5);
            \draw (B5) to node[midway, below]{$28$} (B6);
            \draw[dashed] (B6) to node[midway, below]{$33$} (B7);
            \draw (B7) to node[midway, below]{$1$} (B8);
            \draw (B8) to node[midway, below]{$1$} (B9);
            \draw[dashed] (B9) to node[midway, below]{$35$} (B10);
            \draw (B10) to node[midway, below]{$1$} (B11);
            \draw (B11) to node[midway, below]{$15$} (B12);
            \draw[dashed] (B12) to[out=-110,in=-70, looseness=0.3] node[above]{$69$} (B1);

            \draw (A3) to node[pos=0.94, xshift=2pt,yshift=5pt]{$1$} (B6);
            \draw (A2) to node[pos=0.87, xshift=2pt,yshift=5pt]{$1$} (B11);
            \draw (A1) to node[pos=0.2, xshift=-4pt,yshift=-4.5pt]{$1$} (B8);
            \draw (A6) to node[pos=0.7, xshift=-4pt,yshift=4pt]{$1$} (B3);
            \draw (A5) to node[pos=0.3, xshift=-1pt,yshift=6pt]{$1$} (B10);
            \draw (A4) to node[pos=0.12, xshift=-4pt,yshift=-4pt]{$1$} (B9);
            \draw (A7) to node[pos=0.87, xshift=0pt,yshift=7pt]{$1$} (B1);
            \draw (A8) to node[pos=0.2, xshift=-5pt,yshift=4pt]{$1$} (B5);
            \draw (A9) to node[pos=0.87, xshift=-1pt,yshift=6pt]{$1$} (B7);
            \draw (A10) to node[pos=0.87, xshift=-4pt,yshift=5pt]{$1$} (B2);
            \draw (A11) to node[pos=0.9, xshift=-6pt,yshift=5pt]{$12$} (B4);
            \draw (A12) to node[pos=0.5, xshift=-7pt,yshift=0pt]{$1$} (B12);
		\end{tikzpicture}
    \caption{Example of a weighted configuration graph on an ordered bipartite graph with $H=Q_3$, for which there does not exist any good long cycle. The weight edges are specified as edge labels. Dangerous edges are dashed.}
    \label{fig:hypercubebadexample}
\end{figure}
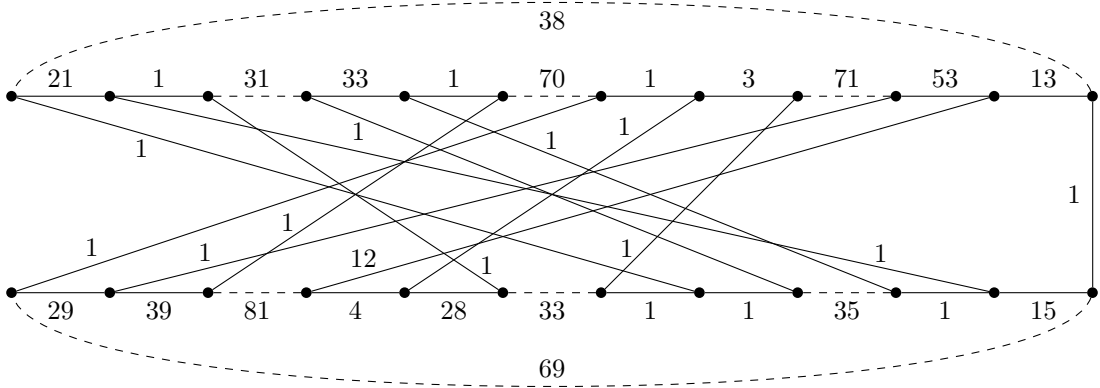

In the course of our research we observed that several other well-known open problems can be directly linked to lower bounds for the longest paths and cycles in connected vertex-transitive graphs. We believe it is worth recording these relationships here, as they might spark further progress on the problem in the future. 

We first need some notation. Given a graph $G$, a \emph{longest path transversal} in $G$ is defined as a set of vertices $S$ such that $S$ intersects every longest path in $G$. Denote by $\mathrm{lpt}(G)$ and $\mathrm{lct}(G)$ the minimum size of a longest path and cycle transversal in $G$, respectively. Furthermore, let us denote by $\mathrm{mipc}(G)$ and $\mathrm{micc}(G)$ respectively, the largest number $k\ge 1$ such that for every collection of $k$ longest paths (cycles, respectively) in the graph $G$, there exists a vertex that is common to all of these $k$ paths (cycles). If all longest paths (cycles) in $G$ meet in a single vertex, then we set $\mathrm{mipc}(G)=\infty$ or $\mathrm{micc}(G)=\infty$, respectively. We observe the following lower bounds on longest cycles and paths in vertex-transitive graphs in terms of these parameters. The two lower bounds in terms of $\mathrm{lpt}(G)$ and $\mathrm{lct}(G)$ were already stated in a slightly different language in the paper of DeVos~\cite{deVos}.
\begin{proposition}\label{prop:lowerbounds}
If $G$ is a connected vertex-transitive graph on $n\ge 3$ vertices, then $G$ contains a path of length at least $\frac{n}{\mathrm{lpt}(G)}-1$ and a cycle of length at least $\frac{n}{\mathrm{lct}(G)}$. Furthermore, $G$ contains a path of length at least $n^{1-1/\mathrm{mipc}(G)}-1$ and a cycle of length at least $n^{1-1/\mathrm{micc}(G)}$. 
\end{proposition}
\begin{proof}
It is easy to observe that $G$ must in fact be $2$-connected: Removing a leaf of a spanning tree of $G$ leaves a connected graph, and since $G$ is vertex-transitive, this means that remvoving \emph{any} vertex of $G$ leaves a connected graph. Hence, by Observation~\ref{obs:intersect} any two longest paths and any two longest cycles in $G$ intersect each other. A key observation is that if a mapping $\phi:V(G)\rightarrow V(G)$ is drawn uniformly at random from the set of all automorphisms of $G$, then, since $G$ is vertex-transitive, for every fixed vertex $u \in V(G)$, its image $\phi(u)$ is uniformly distributed among all vertices of $G$. 

The proofs of the two statements for paths and cycles are completely analogous, which is why in the following we only give the arguments for the two lower bounds on the maximum length of a path. Fix some longest path $P$ in $G$, as well as minimum-size longest path transversal $S$. Let $\phi$ be a random automorphism of $G$. Then clearly, the image $\phi(P)$ is also longest path in $G$ and thus must be hit by $S$. This implies
$$1\le \mathbb{E}[|V(\phi(P))\cap S|]=\sum_{v\in S}\mathbb{P}(v \in \phi(V(P)))= \sum_{v\in S}\sum_{u \in V(P)}{\mathbb{P}(\phi(u)=v)}=\frac{|S||V(P)|}{n}.$$ Rearranging now yields the desired lower bound $|P|=|V(P)|-1\ge \frac{n}{\mathrm{lpt}(G)}-1$.

Let us now move on to the second lower bound. Let $k=\mathrm{mipc}(G)$. Then every collection of $k$ longest paths in $G$ meet in a common vertex. Let $\phi_1,\ldots,\phi_k$ be $k$ independently and uniformly drawn random automorphisms of $G$. Then the images $\phi_1(P),\ldots,\phi_k(P)$ form a collection of $k$ longest paths in $G$ and thus must meet in a common vertex. This implies
$$1\le \mathbb{E}\left[\left|\bigcap_{i=1}^{k}V(\phi_i(P))\right|\right]=\sum_{v \in V(G)}\mathbb{P}\left[\bigwedge_{i=1}^{k}\{v\in V(\phi_i(P))\}\right]= \sum_{v\in V(G)}\prod_{i=1}^{k}\mathbb{P}[v\in V(\phi_i(P))]$$ $$= \sum_{v\in V(G)}\prod_{i=1}^{k}{\sum_{u \in V(P)}\mathbb{P}[\phi_i(u)=v]}= n\cdot \left(\frac{|V(P)|}{n}\right)^k=\frac{|V(P)|^k}{n^{k-1}}.$$ Rearranging now yields $|V(P)|\ge n^{1-1/k}$ and thus $|P|\ge n^{1-1/\mathrm{mipc}(G)}-1$, as claimed. This concludes the proof.
\end{proof}
Given the lower bounds on the length of longest cycles from Proposition~\ref{prop:lowerbounds}, it is natural to ask for the state of the art on bounds for the parameters $\mathrm{lpt}(G)$, $\mathrm{lct}(G)$, $\mathrm{mipc}(G)$,
$\mathrm{micc}(G)$, and we conclude with a discussion of these. 

A famous question asked by Gallai~\cite{gallai} in 1966 is whether $\mathrm{lpt}(G)=1$ for every connected graph $G$, i.e., whether all longest paths meet in one vertex. This was answered in the negative by Walther~\cite{walther}, who constructed a connected graph $G$ on $25$ vertices with $\mathrm{lpt}(G)=2$. A smaller counterexample on $12$ vertices was found independently by Walther and Zamfirescu~\cite{walther2,zamfirescu1976}. Later, Gr\"{u}nbaum~\cite{grunbaum} constructed a connected graph $G$ on $324$ vertices with $\mathrm{lpt}(G)=3$, and amazingly, this remains the best known lower bound for connected graphs to this date. In particular, it is still an open question, raised by both Walther and Zamfirescu~\cite{zamfirescu1972}, whether there exists an absolute constant $c>0$ such that $\mathrm{lpt}(G)\le c$ for every connected graph $G$. Note that this would immediately yield a lower bound of $\Omega(n)$ on the maximum length of paths in vertex-transitive graphs. The question of Walther and Zamfirescu is currently out of reach, and in fact, all the known general upper bounds for longest path and cycle transversals in connected  and $2$-connected graphs grow polynomially in $n$, see the recent sequence of papers~\cite{kiersteadren,long,rautenbachsereni} on this problem. The current state of the art bounds are $\mathrm{lpt}(G)\le 5n^{2/3}$ for all connected $n$-vertex graphs and $\mathrm{lct}(G)\le 5n^{2/3}$ for all $2$-connected $n$-vertex graphs by Kierstead and Ren~\cite{kiersteadren}. Unfortunately, these only provide lower bounds of order $\Omega(n^{1/3})$ on longest paths and cycles in vertex-transitive graphs, but improving the upper bounds on $\mathrm{lpt}(G)$ and $\mathrm{lct}(G)$ to $n^{o(1)}$, for instance, would be very interesting. 

Let us now turn to the parameters $\mathrm{mipc}(G)$ and $\mathrm{micc}(G)$. A well-known and longstanding conjecture, popularized by Zamfirescu since the 80s, asks whether every three longest paths in a connected graph must intersect in a common vertex, that is, whether $\mathrm{mipc}(G)\ge 3$ for every connected graph, see e.g.~\cite{axenovich,cerioli,chen17,derezende,Shabbir13,voss,zamfirescuagain}. By Proposition~\ref{prop:lowerbounds}, the truth of this conjecture would directly yield an improved lower bound of $\Omega(n^{2/3})$ for longest paths in vertex-transitive graphs. Similarly, it remains open whether $\mathrm{micc}(G)\ge 3$ for every $2$-connected graph $G$~\cite{Shabbir13}, which would give the same improved lower bound for longest cycles in vertex-transitive graphs. In fact, it is even possible that $\mathrm{mipc}(G)\ge 6$ for every connected $G$ and $\mathrm{micc}(G)\ge 6$ for every $2$-connected $G$, which would yield much improved lower bounds of $\Omega(n^{5/6})$ for longest paths and cycles in vertex-transitive graphs.

We believe that because of this, the problems of bounding the parameters $\mathrm{lpt}(G)$, $\mathrm{lct}(G)$, $\mathrm{mipc}(G)$, $\mathrm{micc}(G)$ on vertex-transitive graphs deserve further attention.

\subsubsection*{Acknowledgements} The fifth author would like to thank Matt DeVos for introducing  her to Lov\'asz's conjecture during the 2022 Barbados Graph Theory workshop. The second and fifth authors would also like to  thank Matija Bucic for many fruitful discussions on Lov\'asz's conjecture.  Part of the work leading to this paper was conducted during the 2024 Barbados Graph Theory Workshop. We would like to thank the organisers, Sergey Norin, Paul Seymour, and David Wood, as well as the other participants for creating a productive atmosphere. 

\bibliographystyle{abbrvurl}
\bibliography{refs}

\end{document}